\documentclass[letter, 11pt]{article}
\usepackage[english]{babel}
\usepackage[utf8]{inputenc}

\usepackage[margin = 1.2 in, top = 1 in, bottom = 1.2 in]{geometry}
\makeatletter
\g@addto@macro\normalsize{%
  \setlength\abovedisplayskip{7pt}
  \setlength\belowdisplayskip{7pt}
  \setlength\abovedisplayshortskip{7pt}
  \setlength\belowdisplayshortskip{7pt}
}
\makeatother

\usepackage{graphicx}
\usepackage{cancel}

\usepackage{enumerate}
\usepackage{scalerel,stackengine}
\usepackage{graphics}
\usepackage{fancyhdr}
\usepackage{cancel}
\usepackage{enumerate}
\usepackage{amssymb}
\usepackage{amsmath}
\usepackage{hyperref}
\usepackage{mdwlist}
\usepackage{etoolbox}
\usepackage{latexsym}
\usepackage{amsthm}
\usepackage{multicol}
\usepackage{makeidx}
\usepackage{bm}
\usepackage{dsfont}
\usepackage{graphicx}
\usepackage{wrapfig}
\usepackage{multicol}
\usepackage{makeidx}
\usepackage{bm}
\usepackage{dsfont}
\usepackage{mdframed,color}
\usepackage{amsmath,amssymb}
\usepackage[dvipsnames]{xcolor}
\usepackage{mathtools}

\usepackage[capitalize]{cleveref}
\usepackage{bbm, dsfont}
\usepackage{tikz}
\usetikzlibrary{matrix}
\usetikzlibrary{positioning}

\usepackage{titlesec}
\titleformat{\section}
  {\large\center\bfseries}
  {\thesection.}{.7em}{}
\titlespacing*{\section}{0pt}{3.5ex plus 0ex minus 0ex}{1.5ex plus 0ex}
\titleformat{\subsection}
  {\center\bfseries}
  {\thesubsection.}{.7em}{}
\titlespacing*{\subsection}{0pt}{3.5ex plus 0ex minus 0ex}{1.5ex plus 0ex}
\titleformat{\subsubsection}
  {\center\bfseries}
  {\thesubsubsection.}{.7em}{}
\titlespacing*{\subsubsection}{0pt}{3.5ex plus 0ex minus 0ex}{1.5ex plus 0ex}

\addto\captionsenglish{}

\usepackage{titling}
\newtheoremstyle{plain}{3mm}{3mm}{\slshape}{}{\bfseries}{.}{.5em}{}
\newtheoremstyle{definition}{2mm}{2mm}{}{}{\bfseries}{.}{.5em}{}
\theoremstyle{plain} 
	
\newtheorem{theorem}{Theorem}[section]

\newtheorem{proposition}[theorem]{Proposition}
\newtheorem{question}[theorem]{Question}
\newtheorem{conjecture}[theorem]{Conjecture}
\newtheorem{lemma}[theorem]{Lemma}
\newtheorem{corollary}[theorem]{Corollary}
\theoremstyle{definition} 
\newtheorem{definition}[theorem]{Definition}
\newtheorem{remark}[theorem]{Remark}

\theoremstyle{plain} 
\newcounter{MainTheoremCounter}

\newtheorem{maintheorem}[MainTheoremCounter]{Theorem}
\newtheorem{maincorollary}[MainTheoremCounter]{Corollary}
\theoremstyle{plain}
\newtheorem*{namedthm}{\namedthmname}
\newcounter{namedthm}
	
\newcommand{\N}{\mathbb{N}}
\newcommand{\Z}{\mathbb{Z}}

\newcommand{\R}{\mathbb{R}}

\renewcommand{\P}{\mathbb{P}}

\newcommand{\bP}{\mathbb{P}}

\DeclareMathOperator\supp{supp}

\usepackage[normalem]{ulem}

\newcommand{\E}{\operatorname{\mathbb{E}}}

\begin{document}
\author{By~~{\scshape Felipe~Hern\'andez}~~and~~{\scshape Luke~Hetzel}}
\date{\small \today}
\title{{\bfseries On growth rates of infinite and finite sumsets}}
\maketitle

\begin{abstract}
We study growth rates of infinite and finite sumset patterns in sets of positive density. In the infinite setting, we show that no such rate exists, answering a question of Kra, Moreira, Ritcher, and Robertson. Namely, for any proposed growth rate $\mathcal{H}: \N \to \N$ tending to infinity, we construct a set $A$ of lower density \(1\) such that whenever $B,C \subseteq \N$ are infinite and $B+C \subseteq A$ we have the minimum of $|B\cap [N]|$ and $|C \cap [N]|$ is less than $\mathcal{H}(N)$ for infinitely many $N$. In the finitary setting, we prove that for all $\delta \in (0,1)$, for all sufficiently large $N$, for all subsets $A$ of \(\{1,\dots,N\}\) of proportion $\delta$, one can always find sumset patterns \(B+C\subseteq A\) with \(|B|\) and \(|C|\) of order \(\log N\), partially resolving a conjecture of Kra, Moreira, Richter, and Robertson. Moreover, we generalize our second result to the case of the $k$-fold sum $B_1 + B_2 + \ldots + B_k \subseteq A$.
\end{abstract}

\section{Introduction}
One of the most famous and celebrated results in Ramsey theory is van der Waerden's theorem \cite{vdWaerden27}: for any integers $r, k$ there exists $N \in \N$ such that if $[N]$ is colored with $r$ colors, then some color class contains an arithmetic progression of length $k$. A limitation of this theorem is that one cannot specify in advance which color must contain these progressions. In 1953 Roth \cite{Roth53} showed that for every $\delta \in (0,1]$, there exists $N \in \N$ such that for every $ A \subseteq [N]$ with $|A| \geq \delta N$, the set $A$ contains an arithmetic progression of length $3$. In 1975 Szemer\'edi \cite{Szemeredi75} proved that for every $k\in \N$ and $\delta \in (0,1]$, there exists $N \in \N$ such that for every $ A \subseteq [N]$ with $|A| \geq \delta N$, the set $A$ contains an arithmetic progression of length $k$, resolving a conjecture of Erd\H{o}s and Tur\'an \cite{Erdos_Turan_seq_intergers:1936}.

A direct application of the compactness principle shows that the results of van der Waerden and Szemer\'edi both have equivalent infinitary formulations. On one hand, equivalent to van der Warden's theorem is the statement that in every finite coloring of $\N$ at least one color contains arbitrarily large arithmetic progressions. On the other hand, Szemer\'edi's theorem is equivalent to the statement that every subset of the $\N$ with positive upper Banach density (see \cref{notation} for the definition) contains arbitrarily large arithmetic progressions.    
 
In 2019, Moreira, Richter, and Robertson \cite{Moreira_Richter_Robertson19} showed that every set of positive upper Banach density contains a set of the form $B+C$ for two infinite sets $B$ and $C$, proving a long-standing conjecture of Erd\H{o}s. Subsequent works have established several generalizations, but since we will mainly focus on configurations of the form $B+C$ in this article, we refer the reader to 
\cite{Charamaras_Mountakis_2025,FH25,HKR25,host2019short,Kousek_2026,Kra_Moreira_Richter_Robertson:2022,Kra_Moreira_Richter_Robertson:2023,kmrr25} for further developments.

One may wonder whether it is possible to say something about the size of $B$ and $C$ in the theorem of Moreira, Richter, and Robertson. In this direction, Host \cite{host2019short} showed that there exists a set of positive density not containing any sum of a set of positive density and an infinite set. However, one can still ask whether the theorem of Moreira, Richter, and Robertson admits an equivalent finitary formulation, as do the theorems of van der Waerden and Szemerédi. Since positive density corresponds to linear growth, one might instead ask whether $B$ and $C$ can be forced to grow at some prescribed sublinear rate. In their survey paper, Kra, Moreira, Richter, and Robertson \cite{Kra_Moreira_Richter_Robertson_problems} pose the following question.

\begin{question}[{\cite[Question 4.9]{Kra_Moreira_Richter_Robertson_problems}}]\label{Infinite_density_question}
    Given $\delta \in (0,1)$, is there a nondecreasing function $\mathcal{H}: \N \to \N$ with $\mathcal{H}(N) \to \infty$ as $N \to \infty$ satisfying the following: for any set $A \subseteq \N$ with $\underline{d}(A) \geq \delta$ there exist $B,C \subseteq \N$ with
        \[
        \min \{ |B \cap [N]|, |C \cap [N]| \} \geq \mathcal{H}(N)
        \]
    for all sufficiently large $N$ and such that $B+C \subseteq A$?
\end{question}
The authors in \cite{Kra_Moreira_Richter_Robertson_problems} remark that not even the coloring analog of \cref{Infinite_density_question} has been solved. Our first result answers both, \cref{Infinite_density_question} and its natural coloring analog.
\begin{maintheorem}\label{No-growth-rate-for-infinite-sumsets}
    Let $\mathcal{H} : \N \to \N$ be a non-decreasing function with $\mathcal{H} (N) \to \infty $ as $N \to \infty$. 
    \begin{enumerate}
    \item\label{Item-1-Theo-A} There is a $2$-coloring of $\N$ such that for all two infinite sets $B,C\subseteq \N$ with $B+C$ monochromatic we have
    $$\max\{|B\cap [N]|, |C\cap [N]|\} < \mathcal{H}(N) $$
    for infinitely many $N\in \N$.
    \item\label{Item-2-Theo-A} There is a set $A \subseteq \N$ with $\underline{d}(A) = 1$ such for all two infinite sets $B,C \subseteq \N$ with $B+C \subseteq A$ we have 
        \[
         \min{ \{|B \cap [N]|, |C \cap [N]| \}}<\mathcal{H}(N)
        \]
        for infinitely many $N\in \N$.
    \end{enumerate}
 
\end{maintheorem}
\begin{remark}
   We are not sure if we can replace $\min$ with $\max$ in (\ref{Item-2-Theo-A}), and if we can replace ``infinitely many $N\in \N$''  with ``cofinite $N\in \N$'' in (\ref{Item-1-Theo-A}) or (\ref{Item-2-Theo-A}) (see \cref{q-1} and \cref{q-2}).
\end{remark}

\cref{No-growth-rate-for-infinite-sumsets} shows that no growth rate can be found for infinite sumsets. However, a similar question can be asked for finite sumsets. For $\delta \in (0,1)$ and $N\in \N$, we define $\phi_{\delta}$ as the largest integer such that any set $A \subseteq [N]$ with $|A| \geq \delta N$ contains a set of the form $B+C$ where $B,C \subseteq [N]$ satisfy $\min\{|B|,|C|\} \geq \phi_{\delta}(N)$. This function captures the asymptotic behavior of the growth rate of finite sumsets that can be found in subsets of $[N]$ with density at least $\delta$. In \cite{Kra_Moreira_Richter_Robertson_problems}, Kra, Moreira, Richter and Robertson made the following conjecture.

\begin{conjecture}[{\cite[Conjecture 4.10]{Kra_Moreira_Richter_Robertson_problems}}]\label{finite-growth-rate-conj}
    Let $\delta \in (0,1).$ Then
    \begin{equation}\label{conj-item-1} 
        0<\liminf_{N \to \infty} \frac{\phi_{\delta}(N)}{\log(N)},
    \end{equation}
    and
    \begin{equation}\label{conj-item-2} 
      \limsup_{N \to \infty} \frac{\phi_{\delta}(N)}{\log(N)} < \infty.  
    \end{equation}
\end{conjecture}
The lower bound over the liminf would show that $\log(N)$ is a growth rate for finite sumsets. Meanwhile, the upper bound over the limsup would show that such a growth rate is ``optimal.'' 
\begin{remark}
  In this work, we focus only on \cref{conj-item-1}, as \cref{conj-item-2} seems out of reach. However, it is worth mentioning that using \cite[Theorem 2]{KONYAGIN201861}, one can easily deduce that for each $\delta \in [0,1/2]$, and for all $\epsilon>0$
  \begin{equation}\label{eq-limsup-modified}
      \limsup_{N \to \infty} \frac{\phi_{\delta}(N)}{\log(N) \log\log(N)^{10+\epsilon}}=0.
  \end{equation}
To our knowledge, this the closest result to \cref{conj-item-2}. 
\end{remark}

The use of the intersectivity lemma of Bergelson \cite{MR809951} has been crucial for finding infinite sumsets $B+C$ inside of sets of positive upper Banach density in multiple previous works (see for example \cite{diNasso_Golbring_Jin_Leth_Lupini_Mahlburg2015sumset,host2019short,Moreira_Richter_Robertson19}). Following such previous developments, in order to prove \cref{conj-item-1} we develop a finitistic version of the intersectivity lemma, which we state next.
\begin{maintheorem}\label{Finitistic-Intersectivity-Lemma}
      Let $\delta\in (0,1)$ and $0<\sigma<1/\log{(\delta^{-1})}$. Let $\alpha>0$ and let $M=M(N) =\omega(N^\alpha)$ for $N\in \N$. Then, there is $\overline{N}\in \N$ such that for all $N\geq \overline{N}$, the following holds: for each $m\in [M]$ let $A_{m}\subseteq [N]$ satisfying $|A_{m}|\geq \delta N$. Then there is $B\subseteq [M] $ with $|B|\geq \sigma \log{N}$ such that 
        $$\left| \bigcap_{m\in B} A_{m} \right| \geq N^{1-\sigma \log(\delta^{-1})}.$$
\end{maintheorem}

As a corollary, our next result proves the first part of \cref{finite-growth-rate-conj}. 

\begin{maincorollary}\label{Finite-Sumsets-Growth-rate}
    For every $\delta \in (0,1)$ we have 
    $$\frac{1}{\log{\frac{1}{\delta}}}\leq \liminf_{N\to \infty} \frac{\phi_{\delta}(N)}{\log{N}}.$$
\end{maincorollary}
We remark that we actually obtain a more general version of \cref{Finitistic-Intersectivity-Lemma} and \cref{Finite-Sumsets-Growth-rate} related to the growth rate for $k$-sumsets $B_1+\cdots+B_k$. For the detailed statement of the general result, we refer the reader to \cref{sec3}. 

\subsection*{Acknowledgments}
We are grateful to Paul Horn, Anh Le, Ronnie Pavlov, and Florian Richter for various helpful comments regarding a previous draft of this article.

\section{Notation}\label{notation}
We write $\N = \{1,2,\ldots \}$ as the set of naturals and $\N_0 = \{0,1,2,\ldots \}$ to be the set of non-negative integers. For $N\in \N$ and $M \in \N_0$ we denote $[N]:=\{1,\ldots,N\}$ and $[M,N)=\{M,M+1,\ldots,N-1\}$. 

\begin{definition}For a set $A \subseteq \N$
 define the \emph{lower natural density} and \emph{upper natural density} of $A$ by $$\underline{d}(A):=\liminf_{N \to \infty} \frac{|A \cap [N]|}{N} \text{ and }\overline{d}(A):=\limsup_{N \to \infty} \frac{|A \cap [N]|}{N}$$ respectively. Analogously, we define the \emph{upper Banach density} by $$d^*(A):=\limsup_{N-M \to \infty} |A \cap [M,N)|/(N-M).$$
\end{definition}

In order to construct the set $A$ in \cref{No-growth-rate-for-infinite-sumsets} (\ref{Item-2-Theo-A}) we make use of the notion of \textit{Schnirelmann density} which is usually defined in infinite sets, but we extend its definition to finite sets in the following way.  
\begin{definition}
    Let $N\in \N$. We define the \textit{Schnirelmann density} of a set $A\subseteq [0,N)$ as 
    $$\sigma_N(A)= \min_{k\leq N} \frac{|A\cap [0,k)|}{k}. $$
\end{definition}

\begin{definition}
Let $w \in \{0,1\}^{[0,n)}$ be a finite word, we define $|w| = n$ to be the length of $w$. Additionally for $i \in \N$ we define $w^i = ww\ldots w$ to be the concatenation of $w$ with itself $i$ times. Finally we define $\supp (w) = \{ i \in [0,n): w_i = 1 \}$
\end{definition}

\begin{definition}
    For eventually positive functions $f,g:\N\to \R$, we say $f(n)=\omega(g(n))$ if $f(n)/g(n)\to \infty$ as $n\to\infty$. For functions $f,g:\N\to \R$ such that $g(n)\neq 0$ for $n$ sufficiently large , we say $f(n)\thicksim g(n)$ if $f(n)/g(n)\to 1$ as $n\to \infty$. 
\end{definition} 

\begin{definition} Let $N \in \N$ and $\alpha \in \R$ with $\alpha \leq N$. Define $\binom{[N]}{\alpha}$ to be the set of all subsets of $[N]$ of cardinality $\lceil \alpha \rceil$.
\end{definition}

\section{Growth Rates For Infinite Sumsets}
\subsection{Growth rate problem for colorings}
In this section, we prove \cref{No-growth-rate-for-infinite-sumsets}. We first show that for every possible growth rate, $\mathcal{H}$, there is a 2-coloring of the natural numbers such that every infinite sumset fails to follow $\mathcal{H}$ infinitely many times. This example contains in embryonic form the ideas that we later use to construct the analogous counterexample for the density version of the problem. 

Since part of our proof of \cref{No-growth-rate-for-infinite-sumsets} (\ref{Item-2-Theo-A}) will involve projecting into $\Z_N$ it becomes notationally simpler to work in $\N_0$ instead of $\N$. In order to keep notation consistent between the coloring and density problem, we similarly adapt the proof of \cref{No-growth-rate-for-infinite-sumsets} (\ref{Item-1-Theo-A}). Since \cref{No-growth-rate-for-infinite-sumsets} is equivalent in either $\N$ or $\N_0$ by shift invariance, this does not lose any generality.
\begingroup
\def\thetheorem{\ref{No-growth-rate-for-infinite-sumsets} \ref{Item-1-Theo-A}}
\begin{theorem}\label{Coloring-problem}
    For any non-decreasing function $\mathcal{H}:\N\to \N$ with $\mathcal{H}(N) \to \infty$ as $N\to \infty$ there is a $2$-coloring of $\N_0$ such that for all two infinite sets $B,C\subseteq \N$ with $B+C$ monochromatic we have
    \begin{equation}\label{eq-colorings}
       \max\{|B\cap [N]|, |C\cap [N]|\} < \mathcal{H}(N)  
    \end{equation}
    for infinitely many $N$.
\end{theorem}
\addtocounter{theorem}{-1} 
\endgroup
\begin{proof}
    Assume by contradiction that there is a non-decreasing $\mathcal{H}:\N\to \N$ with $\mathcal{H}(N) \to \infty$ as $N\to \infty$ such that for every $2$-coloring of $\N$ there are two infinite sets $B,C\subseteq \N_0$ such that $B+C$ is monochromatic and for all sufficiently large $N\in \N$ we have 
       $$\max\{|B\cap [N]|, |C\cap [N]|\} \geq \mathcal{H}(N) .$$
    Define the function $I:\N\to \N$ as 
    $$I(n)= \min\{ r\in \N: \mathcal{H}(r)>n\}. $$
       
       We will build a $2$-coloring $\N_0=A_0\sqcup A_1$ that witnesses the failure of $H$ by building the indicator function of the color $A_1$ as the limit of a sequence of finite words $(w_i)_{i\in \N_0}$. To define such sequence, we set $w_1=01$ and for $i\in \N_0$ we define
       $$w_{i+1}= \begin{cases}
           w_i1^{2(I(|w_i|)+|w_i|)+1} & \text{ if }w_i \text{ ends with }0 \\
           w_i0^{2(I(|w_i|)+|w_i|)+1}& \text{ if }w_i \text{ ends with }1
       \end{cases}.  $$ 
     Let $W \in \{0,1\}^{\N_0}$ be the infinite word such that for each $i\in \N$, $W|_{[|w_i|]}=w_i$, i.e. $W$ starts with the subword $w_i$. Define $A_1=\{ i\in \N_0 \mid W(i)=1\} $ and $A_0=\{ i\in \N_0 \mid W(i)=0\}$. We observe that $(A_0,A_1)$ defines a coloring of $\N_0$. Let $B,C\subseteq \N_0$ be infinite sets such that $B+C$ is monochromatic for the coloring $(A_0,A_1)$. Without loss of generality, assume that $B+C\subseteq A_0$, since the case $B+C\subseteq A_1$ is analogous. Let $N'\in \N$ be such that for all $N\geq N'$,  
      $$\max\{|B\cap [N]|, |C\cap[N]|\}> \mathcal{H}(N). $$
      Let $i\in \N$ such that $|w_i|\geq N'$ and such that $w_i$ ends with $0$. As $w_{i+1}$ extends $w_i$ by only adding $1$'s, we have that by the pingeonhole principle, $\max(|B\cap [|w_{i+1}|]|,|C\cap [|w_{i+1}|]|)\leq |w_i|$. Otherwise, if for example $|C\cap [|w_{i+1}|]|>|w_i|$, then $B+C$ must contain an element in the interval $\{|w_i|+1,\ldots,|w_{i+1}|\}$ which is contained in $A_1$, a contradiction.
      
      On the other hand, we have that $$\mathcal{H}(|w_{i+1}|)\geq \mathcal{H}(I(|w_i|))>|w_i|\geq \max(|B\cap [|w_{i+1}|]|,|C\cap [|w_{i+1}|]|). $$
      Since $|w_{i+1|}\geq |w_i|\geq N'$, we contradict the hypothesis over $\mathcal{H}$. 
\end{proof}
\begin{remark}\label{remark-max-min}
    In \cref{eq-colorings} we are able to take maximum instead of minimum. However, due to the construction both colors $A_0$ and $A_1$ have lower density equal to $0$. Hence neither set forms a counterexample to the analogous density problem.
\end{remark}
\subsection{Growth rate problem for sets of positive lower density}
As explained by \cref{remark-max-min}, the example provided by \cref{Coloring-problem} is not good enough to tackle \cref{Infinite_density_question}. We use the required growth rate to force a shift of the beginning of $A_1$ to be included in $A_0$ and force a shift of the beginning of $A_0$ to be included in $A_1$, and continue repeating the process. This says that any monochromatic $B,C$ that are satisfying the growth rate will be forced into the wrong color later on. The main issue is that in order to control the cardinalities of $B$ and $C$, we forbid potentially super-exponentially increasing intervals for each color. This forces each color to have lower density zero.

The idea to fix this will be to forbid specific patterns instead of a full interval. We prove that for any finite sufficiently large finite pattern, we can find a set with high density that does not include the pattern. We can then iteratively forbid all possible patterns that could have appeared in a candidate for $B$ which satisfies the growth rate.  

The key proposition that will allow us to forbid specific patterns is the following. \begin{proposition}\label{Forbidding-finite-patterns}
    For every $\epsilon>0$ there exists $K\in \N$ such that for every finite subset $P\subseteq \N_0$ with $|P|\geq K$ there exist arbitrarily large $N\in \N$ and $A\subseteq [0,N)$ with $\sigma_N(A)\geq 1- \epsilon$ satisfying $A$ does not contain any shift of $P$ in $\Z_N$.
\end{proposition}

 Now we assume \cref{Forbidding-finite-patterns} to prove \cref{Item-2-Theo-A} of \cref{No-growth-rate-for-infinite-sumsets}, which we restate now for the convenience of the reader. 
 \begingroup
 \def\thetheorem{\ref{No-growth-rate-for-infinite-sumsets} \ref{Item-2-Theo-A}}
\begin{theorem}
    Let $\mathcal{H} : \N \to \N$ be a non-decreasing function with $\mathcal{H} (N) \to \infty $ as $N \to \infty$. Then there exists a set $A \subseteq \N_0$ with $\underline{d}(A) = 1$ such that if $B,C \subseteq \N_0$ are infinite with $B+C \subseteq A$, then there exist infinitely many $N \in \N$ for which 
        \[
         \min{ \{|B \cap [0,N)|, |C \cap [0,N)| \}}<\mathcal{H}(N).
        \]
\end{theorem}
\addtocounter{theorem}{-1} 
\endgroup
\begin{proof}
    Let $\mathcal{H}: \N \to \N$ be non-decreasing with $\mathcal{H}(n) \to \infty $ as $n \to \infty$. Define the function $I:\N\to \N$ as 
    $$I(n)= \min\{ r\in \N: \mathcal{H}(r)>n\}. $$
    Similarly to \cref{Coloring-problem}, we will define a set $A$ of lower density 1 by building its indicator function as concatenation of a sequence of finite words $(w_i)_{i\in \N}$, i.e. $A$ is defined as
    $$A=\{n\in \N_0: W(n)=1\}, $$
    where $W\in \{0,1\}^{\N_0}$ is the unique infinite word such that $W|_{[0,|W_i|)}= W_i$ for each $i\in \N$, where $W_i=w_1\cdots w_i$. These finite words will have the following properties:
    \begin{enumerate}
        \item\label{word-condition-1} (Forbidding patterns) If $l\geq 2$ and $P\subseteq [0,|W_{\ell-1}|)$ with $|P|=\mathcal{H} (|W_{\ell-1}|)$, then there exists $  k\in [|W_{\ell-1}|-1, |W_{\ell}|)  $ such that $k+I(k)\leq |W_{\ell}|-1$ and there is no shift of $P$ in $\supp(W_{\ell})\cap \{k,\ldots, k+I(k)\}$.
        \item\label{word-condition-2} (High density) For every $l\geq 1$, $\sigma_{|w_{\ell}|}(\supp(w_{\ell}))\geq \ell/(\ell+1)$.
    \end{enumerate}
    
    Assume that we have found such sequence $(w_i)_{i\in \N}$, set $A=\bigcup_{i\in \N} \supp(W_i)$. By \cref{word-condition-2} we have that $\underline{d}(A)=1$. For the sake of contradiction, assume that there are infinite sets $B,C\subseteq \N_0$ with $B+C\subseteq A$ such that there is $N'\in \N$ satisfying that for all $N\geq N'$: 
    $$\min\{|B\cap [0,N)|,|C\cap [0,N)|\} \geq \mathcal{H}(N).$$
    Let $\ell\in \N$ such that $|W_{\ell-1}|-1\geq N'$. Take $P\subseteq B\cap [0,|W_{\ell-1}|)$ of cardinality $\mathcal{H}(|W_{\ell-1}|)$. By \cref{word-condition-1} there exists $|W_{\ell-1}|-1\leq k\leq |W_{\ell}|-1$ such that $k+I(k)\leq |W_{\ell}|-1$ and such that no shift of $P$ is present in $\supp(W_{\ell})\cap \{k,\ldots, k+I(k)\}$. Since $P+C\subseteq A$, this implies that $|C\cap [0,I(k))|\leq k$, which in turn implies that
    $$k\geq  \min\{|B\cap [0,I(k))|,|C\cap [0,I(k))|\}\geq  \mathcal{H}(I(k))>k, $$
    a contradiction. This way, such $B$ and $C$ cannot exist.

    Now we prove that such sequence of words $(w_i)_{i\in \N}$ exists. Applying \cref{Forbidding-finite-patterns} with $\epsilon=\frac{1}{2} $ we obtain $K_1\in \N$ such that the conclusion of \cref{Forbidding-finite-patterns} holds. Let $N_1=\max(I(K_1),K_1)$ and define $w_1=1^{N_1}$. Clearly $w_1$ satisfies \cref{word-condition-2}. It also satisfies conditions \cref{word-condition-1} by vacuity. Assume that for $\ell\geq 2$ we have defined $w_{\ell-1}$ satisfying \cref{word-condition-1} and \cref{word-condition-2}. We define $w_{\ell}$ inductively. We start with $w_{\ell}=\varnothing$. For each $P\subseteq [0,|W_{\ell-1}|)$ with cardinality $\mathcal{H}(|W_{\ell-1}|)$ we can use \cref{Forbidding-finite-patterns} to obtain $N>I(|W_{\ell-1}w_{\ell}|)$ and a set $S\subseteq [0,N)$ such that $\sigma_{N}(S)\geq 1- \frac{1}{l+1}$ and $S$ contains no shift of $P$. Setting $v_S\in \{0,1\}^{N}$ such that $S=\supp(v)$, we replace $w_{\ell}$ with $w_{\ell}v_S$. This process finishes in exactly $\binom{|W_{\ell-1}|}{\mathcal{H}(|W_{\ell-1}|)}$ steps. The finite word $w_{\ell}$ obtained satisfies \cref{word-condition-1} by exhaustion of all possible subsets of $[|W_{\ell-1}|]$ with cardinality $\mathcal{H}(|W_{\ell-1}|)$. Moreover, \cref{word-condition-2} is also satisfied as we only concatenate sets with Schnirelmann density greater or equal than $1 - \frac{1}{\ell+1}$. This finishes the proof. 
 \end{proof}

The methods involved in the proof of both the density and coloring versions of \cref{No-growth-rate-for-infinite-sumsets} involve constructing a successive sections of the set $A$ where the growth rate must fail somewhere. This method allows us to build a counterexample to the question, though it is possible there exists a stronger counterexample. We ask the following.

\begin{question}\label{q-1}
    Let $\mathcal{H}:\N \to \N$ be a non-decreasing function with $\mathcal{H}(N) \to \infty$ as $N \to \infty$. Is it true there exists a set $A \subseteq \N$ with $\underline{d}(A) = 1$ such that for every pair of infinite sets $B,C \subseteq \N$ with $B+C \subseteq A$ we have
        \[
        \max \{ |B \cap [N]|, |C \cap [N]| \} < \mathcal{H}(N)
        \]
    for all sufficiently large $N$.
\end{question}
\begin{question}\label{q-2}
    Let $\mathcal{H}:\N \to \N$ be a non-decreasing function with $\mathcal{H}(N) \to \infty$ as $N \to \infty$. Is it true there exists a 2-coloring of $\N$ such that for every pair of infinite sets $B,C \subseteq N$ with $B+C$ monochromatic we have
        \[
        \max \{ |B \cap [N]|, |C \cap [N]| \} < \mathcal{H}(N)
        \]
    for all sufficiently large $N$.
\end{question}

\subsection{Proof of \cref{Forbidding-finite-patterns}}
We will find an $N$ and $A \subseteq [0,N)$ such that $A$ does not contain any shifts of $P \mod N$. Notice that we can simply concatenate the copies of $A$ as subsets of $[0,N), [N, 2N), [2N, 3N),$ etc; in order to find infinitely many values of $N,A$ with $\sigma_N (A) > 1- \epsilon$ and $A$ containing no shifts of $P$. 
We will need the following technical lemma.

\begin{lemma}\label{partition}
    Let $M, r \in \N$ with $M \geq 2$. Then there exists a $K \in \N$ such that for every sequence of $K$ not necessarily distinct natural numbers $p_1, \ldots p_K$, there is a partition $1 \leq i_0   <i_1 < \ldots < i_{r+1} \leq K$ such that either
    \begin{equation}\label{red}
          M \cdot  \sum_{j =i_{\ell-1}}^{i_{\ell}-1} p_j \leq \sum_{j =i_{\ell}}^{i_{\ell+1}-1} p_j, \  \forall \ell\in [r], 
    \end{equation}
        or 
        \begin{equation}\label{blue}
             M  \cdot  \sum_{j =i_{\ell}}^{i_{\ell+1}-1}p_j \leq \sum_{j =i_{\ell-1}}^{i_{\ell}-1} p_j, \    \forall \ell\in [r].
        \end{equation}
\end{lemma}

\begin{proof}
Let us first prove the statement for $M=1$. Let $K= R_{3}(r+2,2)$ the Ramsey number for $3$-subsets with $2$ colors. We color each $3$-subset of $[K]$ such that for $i_1<i_2<i_3$, the set $\{i_1,i_2,i_3\}$ is colored red if 
$$\sum_{j=i_1}^{i_2-1}p_j \leq \sum_{j=i_2}^{i_3-1} p_j$$
and blue otherwise. By Ramsey's theorem, there is a set $I:=\{i_{\ell}\}_{\ell=0}^{r+1}\subseteq [K]$ such that either every $3$-subset of $I$ is red or every $3$-subset of $I$ is blue. On one hand, if the color is red, then we obtain \cref{red}. On the other hand, if the color is blue, then get \cref{blue}, concluding the case $M=1$.

For the general case, take $r=  M^{r'+1}+1$, apply the first case and assume without loss of generality that 
$$    \sum_{j =i_{\ell-1}}^{i_{\ell}-1} p_j \leq \sum_{j =i_{\ell}}^{i_{\ell+1}-1} p_j, \  \forall \ell\in [r].  $$
Define for each $\ell\in \{0,\ldots,r'+1\}$ define $k_{\ell}= \sum_{j=0}^{\ell-1}M^j$. This way, we get that 
$$       M \cdot  \sum_{j =k_{\ell-1}}^{k_{\ell}-1} p_j \leq \sum_{j =k_{\ell}}^{k_{\ell+1}-1} p_j, \  \forall \ell\in [r'],  $$
concluding.
\end{proof}
From \cref{partition} we achieve the following corollary.

\begin{corollary}\label{gaps_exist}
    Let $M, r \in \N$ with $M \geq 2$. Then there exists a $K \in \N$ such that for every $A \subseteq \N_0$ with $|A| = K$ there exists $a_0< a_1< \ldots < a_r \in A$ with either
        \[
    M(a_{i}-a_{i-1}) \leq (a_{i+1} - a_{i}), \ \forall i\in [r-1],
        \]
        or 
        \[
    M(a_{i+1} - a_{i}) \leq (a_{i}-a_{i-1}),\ \forall i\in [r-1].
        \] 
\end{corollary}
\begin{proof}
    Let $K'\in \N$ given by \cref{partition} with $M$ and $r+1$. Set $K=K'+1$. Write in increasing order $A=\{b_i\}_{i=1}^K$, and set $p_j=b_{j+1}-b_{j}$ for $j\in [K-1]$. By the conclusion of \cref{partition} we have that there is a partition $1 \leq i_0   <i_1 < \ldots < i_{r+1} \leq K$ such that either \cref{red} or \cref{blue} hold. Observing that 
    $$\sum_{j=i_{l-1}}^{i_l-1}p_j= \sum_{j=i_{l-1}}^{i_l-1} b_{j+1}-b_j= b_{i_l}-b_{i_{l-1}}, $$
    and setting $a_l=b_{i_l}$ for each $l\in \{0,\ldots,r\}$ yields the conclusion.
\end{proof}
The motivation behind \cref{gaps_exist} is that if a set is large enough, we can always obtain a subset with gaps geometrically increasing. Such kind of patterns are manageable to forbid because they can be treated approximately like an arithmetic progression, which is the content of the following lemma. 
\begin{lemma}[Forbidding Increasing Gap Sequences]\label{Forbidding-increasing-gaps}
    Let $\epsilon> 0 $ and $r\in \N$. There exists $M \in \N$ such that if $B = \{b_0 < \ldots< b_r\} $ with either $$\forall i\in [r-1], \ M(b_{i}-b_{i-1}) \leq (b_{i+1}-b_i) \text{ and } b_1-b_0 \geq M $$ or $$\forall i\in [r-1], \ M(b_{i+1}-b_i) \leq (b_{i}-b_{i-1}) \text{ and } b_r - b_{r-1} \geq M ,$$ then for infinitely many $N$ there is a set $A \subseteq \Z_N$ with $$\sigma_{N}(A) \geq \left(\frac{r}{r+1}\right) (1 - \epsilon),$$ such that $A$ contains no shift of $B$ in $\Z_N$. 
\end{lemma}
\begin{proof}
    Without loss of generality take $B=\{0 = b_0 < b_1< \ldots < b_r\} $ such that $M(b_{i}- b_{i-1}) < b_{i+1} - b_i$ for each $i\in [r-1]$. Our goal is to construct a set $A \subseteq [0,b_1 + b_r)$ such that $A$ contains no shifts of $B $ mod $b_1 + b_r$. It suffices to construct a set $A$ that contains no shifts of the set $B_1:=\{0, b_1, b_1+b_1, b_2 + b_1, \ldots ,b_{r-1} + b_1\}$ mod $b_1 + b_r$. For each $0 \leq i \leq r-2$ let $\ell^{(i)}$ be such that for some $0 \leq k_i < (b_{r-i-1} + (i+2)b_1)$ we have
        \[
        (\ell^{(i)} +1)(b_{r-i-1} + (i+2)b_1) + k_i = b_{r-i} + (i + 1)b_1.
        \]
    Namely, $\ell^{(0)}$ is such that such that $(\ell^{(0)} + 1)(b_{r-1} + 2 b_1) + k_0 = b_r + b_1$ for some $0 \leq k_0 < b_{r-1} + 2b_1$. We will partition $[0, b_1 + b_r )$ into the following $\ell^{(0)} + 1$ many pieces in the following way:
    \[ [0, b_r + b_1 )=  \left(\bigsqcup_{i=1}^{\ell^{(0)}} \left[(i - 1) (b_{r-1} + 2 b_1) , i (b_{r-1} + 2 b_1)\right) \right) \sqcup  \left[\ell^{(0)}(b_{r-1} + 2 b_1), b_r + b_1 \right).\]
    Suppose that there exists $A^1 \subseteq [0,b_{r-1} + 2 b_1)$ such that $A^1$ contains no shifts of $B_1$ mod $b_{r-1} + 2b_1$. Then build $A$ by in each set $[(i-1)(b_{r-1} + 2 b_1) , i(b_{r-1} + 2 b_1))$ including $(i-1)(b_{r-1} + 2 b_1) + A^1$, but not including the elements of $\left[\ell^{(0)}(b_{r-1} + 2 b_1), b_r + b_1 \right)$, this is:
    \[ A=\bigcup_{i=1}^{\ell^{(0)}} \left( (i-1)(b_{r-1} + 2 b_1) + A^1  \right)\]
    
    We claim $A$ contains no shifts of $B$ mod $b_1 + b_r$. We prove this by contradiction and dividing in three cases:

    \textbf{Case 1:} Assume that $A$ contains a shift $j + B $ mod $b_r + b_1$ where $0\leq j \leq b_r+b_1 -1$ and $j + b_r < b_r + b_1$.

    Note first that $|(b_r + b_1 -1) - (j + b_r)|\leq b_1 -1$. Since the length of $\left[\ell^{(0)}(b_{r-1} + 2 b_1), b_r + b_1 \right)$ is greater than or equal to $b_{r-1} + 2b_1$ and $b_1 -1< b_{r-1} + 2b_1$ we have $$\ell(b_{r-1} + 2 b_1) < j +b_r \leq b_r + b_1 -1.$$ Hence, $j + b_r \notin A$ and $j + B \not\subseteq A $ mod $b_r + b_1$, a contradiction.

    \textbf{Case 2:} Assume that $A$ contains a shift $j + B $ mod $b_r + b_1$ where $0\leq j \leq b_r+b_1 -1$, $j + b_r \geq b_r + b_1 $ and $j + b_{r-1} < b_r + b_1$. Then, taking $i = j-b_1 $, $A$ contains a set of the form $$i+B_1=\{i, i+b_1, i + b_1 + b_1, i + b_2 + b_1, \ldots, i + b_{r-1} + b_1\}.$$ This way, either $i + b_{r-1} + b_1$ is in the final partition element $\left[\ell^{(0)}(b_{r-1} + 2 b_1), b_r + b_1 \right)$ or $A^1$ contains a shift of $B_1$ mod $b_{r-1}+2b_1$, a contradiction.

    \textbf{Case 3:} Assume that $A$ contains a shift $j + B $ mod $b_r + b_1$ where $0\leq j \leq b_r+b_1 -1$, $j + b_{r-1} \geq b_r + b_1 $. Then, either $$j+ b_i \in \left[\ell^{(0)}(b_{r-1} + 2 b_1), b_k + b_1 \right),\text{ or }j + b_i \geq b_r + b_1$$ for each $i$, in which case since $b_0 = 0$ we have that $j \geq b_r + b_1 -1$, a contradiction in either case.

    Hence, by exhaustion of all possible cases the result follows. We are left to prove that it is possible to find a set $A^1$ in $[0, b_{r-1} + 2b_1 )$ containing no shift of $B_1$ mod $b_{r-1}+2b_1$. For this, we partition $[0, b_{r-1} + 2b_1 )$ into $\ell^{(1)}+1$ many sets
        \[ [0, b_{r-1}+2b_1 )= \left( \bigsqcup_{i=1}^{\ell^{(1)}}\left[(i -1)(b_{r-2} + 3b_1), i(b_{r-2} + 3b_1)\right) \right) \sqcup   \left[\ell^{(1)}(b_{r-2} + 3b_1), b_{r-1}+ 2b_1 \right).\]
    
    Similarly to the first iteration, it would suffice to construct a $A^2 \subseteq [0, b_{r-2} + 3b_1)$ that contains no shifts of $B_2:=\{0, b_1, 2b_1, 3b_1, b_2 + 2b_1, b_3 + 2b_1, \ldots , b_{r-2} + 2b_1\}$ mod $b_{r-2} + 3b_1$, and then set
        $$A^1:=\bigcup_{i=1}^{\ell^{(1)}} \left( (i-1)(b_{r-2} + 3 b_1) + A^2  \right). $$
 
    Continuing this process in the natural way, in each iteration we partition the interval 
\begin{align*}
    [0, b_{r-j}+(j+1)b_1 )= &\left( \bigsqcup_{i=1}^{\ell^{(j)}}\left[(i -1)(b_{r-(j+1)} + (j+2)b_1), i(b_{r-(j+1)} + (j+2)b_1)\right) \right)\\
    &\sqcup   \left[\ell^{(j)}(b_{r-(j+1)} + (j+2)b_1), b_{r-j}+ (j+1)b_1 \right).
\end{align*}
    to construct a set $A^{j}\subseteq [0,b_{r-j} + (j+1)b_1)$ containing no shifts of $B_j:=\{0,b_1,2b_1,\ldots,(j+1)b_1, b_2+jb_1,\ldots, b_{r-j}+jb_1\}$ mod $b_{b_{r-j} + (j+1)b_1}$ such that  
    $$A^{j}:=\bigcup_{i=1}^{\ell^{(j)}} \left( (i-1)(b_{r-(j+1))} + (j+2) b_1) + A^{j+1}  \right), $$
    where $A^{j+1}$ is yet to be defined in the next iteration.
    
    Consequently, we reduce to the goal of constructing a $A^{r} \subseteq [0, (r+1)b_1 ) $ that contains no shifts of $B_{r}=\{0, b_1, 2b_1, \ldots rb_1\}$. To do so, we set $A^r:=\{0, 1,2,3, \ldots rb_1 -1 \}$. We observe that 
    $$\sigma_{(r+1)b_1}(A^r)= \left( \frac{r}{r+1}\right). $$
    Moreover, since in each step of our proof $A^j$ is construction gluing up $\ell^{(j)}$ copies of $A^{j+1}$ and leaving a small interval empty, we get that for each $j\in [r]:$
    $$ \sigma_{b_{r-j} + (j+1)b_1}(A^j) \geq \sigma_{b_{r-(j+1)} + (j+2)b_1}(A^{j+1}) \left(\frac{\ell^{(j)}}{\ell^{(j)}+2}\right) \geq \left(\frac{r}{r+1}\right)\cdot \min_{i=j,\ldots,r-1}  \left(\frac{\ell^{(i)}}{\ell^{(i)}+2}\right)^{r-1}.$$
    
    We thus have constructed a $A \subseteq [0,b_1 + b_r )$ that contains no shifts of $B$, and such that
        \[\sigma_{b_r + b_1}(A) \geq \left(\frac{r}{r+1}\right)\cdot \min_{i=1,\ldots,r-1}  \left(\frac{\ell^{(i)}}{\ell^{(i)}+2}\right)^{r-1}.\]
    Here note that $\min_{i=1,\ldots,r-1}|\frac{\ell^{(i)}}{\ell^{(i)}+2}|$ can be made arbitrarily close to $1$ by choosing $M$ to be large. Hence, we can choose $M$ such that $$\sigma_{b_r + b_1}(A) \geq \left(\frac{r}{r+1}\right) (1 - \epsilon) ,$$
    concluding.
\end{proof}

Finally, we prove \cref{Forbidding-finite-patterns}.
\begin{proof}[Proof of \cref{Forbidding-finite-patterns}]
    Let $\epsilon>0$ and $r\in\N$. Let $$\delta= 1- \frac{1-\epsilon}{1-\frac{1}{1+r}}$$ and $M\in \N$ given by \cref{Forbidding-increasing-gaps} for $\delta$ and $r$. Then, let $K\in \N$ given by \cref{gaps_exist}, and let $P\subseteq\N$ with $|P|\geq K$. Then, by \cref{gaps_exist}, there is a subset $B=\{b_i\}_{i=0}^r\subseteq P$ such that either $\forall i\in [r], \ M(b_{i}-b_{i-1}) \leq (b_{i+1}-b_i) $ or $\forall i\in [r], \ M(b_{i+1}-b_i) \leq (b_{i}-b_{i-1})$. By \cref{Forbidding-increasing-gaps} applied with $\delta$, $r$, and $B$, we find $N\in \N$ and $A\subseteq \Z_N$ with $|A|\geq (1-\frac{1}{r+1})(1-\delta)N=(1-\epsilon)N$ such that $A$ contains no shift of $B$, concluding the proof.
    \end{proof}

\section{Asymptotics for finite sumsets}\label{sec3}
In this section, we prove \cref{Finitistic-Intersectivity-Lemma} and its more general version. To build intuition, we first prove \cref{Finitistic-Intersectivity-Lemma} before turning to the general result. Finally, we prove \cref{Finite-Sumsets-Growth-rate} by actually providing a growth rate for $k$-fold sumsets $B_1+\cdots+B_k$.
\subsection{Proof of \cref{Finitistic-Intersectivity-Lemma}}
We recall that for $\delta \in (0,1)$ and $N\in \N$, we define $\phi_{\delta}$ as the largest integer such that any set $A \subseteq [N]$ with $|A| \geq \delta N$ contains a set of the form $B+C$ where $B,C \subseteq [N]$ satisfy $\min\{|B|,|C|\} \geq \phi_{\delta}(N)$. Here we restate \cref{Finitistic-Intersectivity-Lemma} for readability.
\begingroup
\def\thetheorem{\ref{Finitistic-Intersectivity-Lemma}}
\begin{theorem}
       Let $\delta\in (0,1)$ and $0<\sigma<1/\log{(\delta^{-1})}$. Let $\alpha>0$ and let $M=M(N) =\omega(N^\alpha)$ for $N\in \N$. Then, there is $\overline{N}\in \N$ such that for all $N\geq \overline{N}$, the following holds: for each $m\in [M]$ let $A_{m}\subseteq [N]$ satisfying $|A_{m}|\geq \delta N$. Then there is $B\subseteq [M] $ with $|B|\geq \sigma \log{N}$ such that 
        $$\left| \bigcap_{m\in B} A_{m} \right| \geq N^{1-\sigma \log(\delta^{-1})}.$$
\end{theorem}
\addtocounter{theorem}{-1} 
\endgroup
For the sake of brevity, we will omit the explanation of some details. These details will be formalized later in the more general proof of \cref{FIL}.
\begin{proof}
    Let $n\in [N]$ be a uniform random variable in $[N]$. Define 
    $$L_n=\{m\in [M] : n\in A_m\}. $$
    By the hypothesis over $A_m$, $m\in [M]$, we have that
    $$\E(|L_n|) = \sum_{m\in [M]} \bP(n\in A_m)\geq \delta M.   $$
    Now, pick $B\in \binom{[M]}{\sigma \log{N}}$ uniformly at random and independent from $n$. 
    \begin{align*}
        \E\left( \left| \bigcap_{m\in B} A_m \right| \right) &= N \E\left(\bP\left(  n\in \bigcap_{m\in B} A_m \Bigg| n\right) \right)\\
        &= N \E\left(\bP\left( B\subseteq L_n  \Bigg| n\right) \right)\\
        &= N \binom{M}{\sigma \log{N}}^{-1} \E\left(  \binom{|L_n|}{\sigma \log{N}} \right) .
    \end{align*}
    Using the fact that the function $\Psi$ defined as
    $$
    \Psi(x) =\begin{cases}
        0 & \text{ if } x<y\\
        \binom{x}{y} & \text{ if }x\geq y
    \end{cases} .
    $$
    is convex for a fixed $y\in [N]$, we get that 
    \begin{align*}
         \E\left(  \binom{|L_n|}{\sigma \log{N}} \right)  \geq    \binom{\E\left(|L_n|\right)}{\sigma \log{N}} \geq \binom{ \delta M}{\sigma \log{N}}. 
    \end{align*}
    This way, we have that 
    $$ \E\left( \left| \bigcap_{m\in B} A_m \right| \right)\geq N \binom{M}{\sigma \log{N}}^{-1} \binom{\E\left(|L_n|\right)}{\sigma \log{N}} \geq N \binom{M}{\sigma \log{N}}^{-1}\binom{ \delta M}{\sigma \log{N}}.$$
    For $N$ sufficiently large we obtain that 
    $$\E\left( \left| \bigcap_{m\in B} A_m \right| \right)\geq   N (\delta M/M)^{\sigma \log{N}}=  N^{1-\sigma \log(\delta^{-1})}, $$
    from which the conclusion follows. 
\end{proof}

\subsection{A generalization of \cref{Finitistic-Intersectivity-Lemma}}
In this section, we will make use of the following fact which is a basic consequence of Stirling's formula: If $n=\omega(u^2)$ then
$$\binom{n}{u}\thicksim \left(\frac{ne}{u}\right)^u  (2\pi u)^{-1/2}.$$
In particular, if $n_1,n_2=\omega(u^2)$ then 
\begin{equation}\label{estimate-binomial-coeff}
   \binom{n_1}{u}\binom{n_2}{u}^{-1}\thicksim \left(\frac{n_1}{n_2}\right)^u   . 
\end{equation}
We now state our generalization of \cref{Finitistic-Intersectivity-Lemma}.
 \begin{theorem}\label{FIL}
      Let $\delta,\tilde{\delta}\in (0,1)$ such that $\tilde{\delta}<\delta$, $k\in \N$ and $0<\sigma<1/\log{(\tilde{\delta}^{-1})}$. Let $\alpha>0$ and let $M=M(N) =\omega(N^\alpha)$ for $N\in \N$. There is $\overline{N}\in \N$ such that for all $N\geq \overline{N}$, the following holds: for each $\boldsymbol{m}\in [M]^{k}$ let $A_{\boldsymbol{m}}\subseteq [N]$ satisfying $|A_{\boldsymbol{m}}|\geq \delta N$. Then, there are $B_1,\ldots,B_{k}\subseteq [M] $ with $|B_i|\geq (\sigma \log{N})^{1/k}$ for each $i\in [k]$, such that 
        $$\left| \bigcap_{\boldsymbol{m}\in B_1\times \cdots \times B_{k}} A_{\boldsymbol{m}} \right| \geq N^{1-\sigma \log(\tilde{\delta}^{-1})}.$$
\end{theorem}
To prove \cref{FIL}, we follow a strategy similar to the one used in the proof of \cref{Finitistic-Intersectivity-Lemma}. Let us briefly recall the idea. First, one obtains a lower bound for $\E(L_n)$, where $n\in [N]$ is chosen uniformly at random. Second, using Jensen's inequality, one derives a lower bound for
$$\E\left( \left|\bigcap_{m\in B} A_m\right|\right) $$
in terms of $\E(L_n)$, where $B\subseteq [M]$ is a set of cardinality $\lceil\sigma \log{N}\rceil$ picked at random. The theorem then follows by taking $N$ sufficiently large. In \cref{FIL} this strategy becomes somewhat more involved. We will write $(m_j)_{j\neq i}$ as short for $$(m_1,\ldots,m_{i-1},m_{i+1},\ldots,m_k) \in [M]^{k-1}.$$
For $n\in [N]$ and $(m_j)_{j\neq i}\in [M]^{k-1}$ define
    $$L_{i,(n,(m_j)_{j\neq i})}=\{m\in [M] : n\in A_{m_1,\ldots m_{i-1}, m ,m_{i+1}, \ldots m_k}\}. $$
Pick a uniform random variable $n\in [N]$. Also, pick $B_1,\ldots,B_k\subseteq [N]$ sets of cardinality $\lceil\sigma \log{N})^{1/k}\rceil $ at random. Our first step in this proof is to find a lower bound for
\begin{equation}\label{first-step-FIL}
   \E\left(\left|\bigcap_{(m_1,\ldots,m_{k-1})\in B_1\times \cdots \times B_{k-1}} L_{k,(n,(m_j)_{j\neq k})} \right|\right). 
\end{equation}
 To find a good bound, we want to use the assumption that the sets $A_{\boldsymbol{m}}$ have large density. In order to do so, we will fix the coordinates $m_j$ inductively. For this reason, it is convenient to prove a lower bound for the following sequence of objects: for every $i\in [k]$, $m_{i+1},\ldots,m_k\in [M]$, we give a lower bound for
    \begin{equation*}
         \E\left( \left| \bigcap_{( m_1,\ldots,m_{i-1})\in  B_1\times \cdots \times B_{i-1} } L_{i,(n,(m_j)_{j\neq i })} \right| \right).
    \end{equation*}
Taking $i=k$ then gives the desired bound in \cref{first-step-FIL}.
The second step is to apply Jensen's inequality in order to lower bound
$$\E\left( \left| \bigcap_{\boldsymbol{m}\in B_1\times \cdots \times B_{k}} A_{\boldsymbol{m}} \right| \right)$$
by an expression depending on the quantity in \cref{first-step-FIL}. The conclusion then follows by choosing $N$ sufficiently large.  

Having outlined the overall idea of the proof, we now proceed with the details.

\begin{proof}[Proof of \cref{FIL}.]
We start fixing the parameters formalizing the estimations associated to \cref{estimate-binomial-coeff}. Let $\eta_1,\ldots,\eta_k\in (0,1]$ so that $\tilde{\delta}=\eta_1\cdots\eta_k\delta$ and for sufficiently large $N$, we have that for each $i\in [k]$,
$$     \binom{M}{(\sigma \log(N))^{1/k}}^{-1}\binom{M (\eta_1\ldots\eta_{i-1}\delta)^{(\sigma \log(N))^{(i-2)/k}}}{(\sigma \log(N))^{1/k}}     $$
is greater or equal than
$$ \eta_i \left(  (\eta_1\ldots\eta_{i-1}\delta)^{(\sigma \log(N))^{(i-2)/k}}  \right)^{(\sigma \log(N))^{1/k}}.$$
The fact that such $\eta_1,\ldots,\eta_k$ and $N$ exists comes from the estimate associated to \cref{estimate-binomial-coeff} with $$n_1=M (\eta_1\ldots\eta_{i-1}\delta)^{(\sigma \log(N))^{(i-2)/k}},~n_2=M,~ \text{ and }u=(\sigma \log(N))^{1/k}.$$
Let $i\in [k]$, $n\in[N]$, and $m_1,\ldots,m_{i-1},m_{i+1},\ldots,m_k\in [N]$. Define
    $$L_{i,(n,(m_j)_{j\neq i})}=\{m\in [M] : n\in A_{m_1,\ldots m_{i-1}, m ,m_{i+1}, \ldots m_k}\}, $$
this is, all the coordinates $m=m_i$ such that $n\in A_{m_1,\ldots,m_k}$. 
    
    Pick $B_1\ldots,B_k\in \binom{[M]}{(\sigma \log{N})^{1/k}}$ and $n\in [N]$ uniformly at random and independent. We claim that for every $i\in [k]$, $m_{i+1},\ldots,m_k\in [M]$,
    \begin{equation}\label{SIL-eq-1}
         \E\left( \left| \bigcap_{( m_1,\ldots,m_{i-1})\in  B_1\times \cdots \times B_{i-1} } L_{i,(n,(m_j)_{j\neq i })} \right| \right)\geq M(\eta_1\cdots\eta_{i-1}\delta)^{(\sigma \log{N})^{(i-1)/k}},
    \end{equation}
    where $m_1,\ldots,m_{i-1}$ are dummy variables. In fact, for $i=1$ we take $m_2,\ldots,m_k\in [M]$ and compute:
    \begin{align*}
        \E\left( \left|  L_{1,(n,m_2,\ldots,m_k)} \right| \right)&= \sum_{m_1\in [M]} \P( n \in A_{m_1,\ldots,m_k})\\
        &\geq  M\delta \geq M\delta',
    \end{align*}   
    which ultimately gives \cref{SIL-eq-1} for $i=1$.
    
    Assume that we know the result for $i-1$, where $i\geq 2$. Fix $m_{i+1},\ldots,m_k\in [M]$. We have that
    \begin{align*}
        &\E\left( \left| \bigcap_{(m_1,\ldots,m_{i-1})\in  B_1\times \cdots \times B_{i-1} } L_{i,(n,(m_j)_{j\neq i })} \right| \right) \\
        & = \sum_{m_i\in [M]} \P\left(m_i \in \bigcap_{(m_1,\ldots,m_{i-1})\in  B_1\times \cdots \times B_{i-1} } L_{i,(n,(m_j)_{j\neq i })}\right)\\
        & =\sum_{m_i\in [M]} \E\left(\P\left(B_{i-1}\subseteq  \bigcap_{(m_1,\ldots,m_{i-2})\in  B_1\times \cdots \times B_{i-2} } L_{i,(n,(m_j)_{j\neq i-1 })}\Bigg| n,B_1,\ldots,B_{i-2} \right) \right)\\
        & =\sum_{m_i\in [M]} \binom{M}{(\sigma \log(N))^{1/k}}^{-1}\E\left( \binom{\left| \bigcap_{(m_1,\ldots,m_{i-2})\in  B_1\times \cdots \times B_{i-2} } L_{i,(n,(m_j)_{j\neq i-1 })} \right|}{(\sigma \log(N))^{1/k}}   \right),
    \end{align*}
where the second equality comes from basic properties of the conditional expectation, and in the third equality we used that, conditional to having chosen $n$, $B_1,\ldots,B_{i-2}$, the probability that $B_{i-1}$ is contained in
$$ \bigcap_{(m_1,\ldots,m_{i-2})\in  B_1\times \cdots \times B_{i-2} } L_{i,(n,(m_j)_{j\neq i-1 })} $$
is equal to    
$$\binom{M}{(\sigma \log(N))^{1/k}}^{-1} \binom{\left| \bigcap_{(m_1,\ldots,m_{i-2})\in  B_1\times \cdots \times B_{i-2} } L_{i,(n,(m_j)_{j\neq i-1 })} \right|}{(\sigma \log(N))^{1/k}} , $$
by definition $B_{i-1}$. 

    Define the function $\Psi$ as
    $$
    \Psi(x) =\begin{cases}
        0 & \text{ if } x<y\\
        \binom{x}{y} & \text{ if }x\geq y
    \end{cases} .
    $$
    We observe that $\Psi$ is convex for a fixed $y\in [N]$. Therefore, by Jensen inequality we have that
    \begin{align*}
        &\E\left( \binom{\left| \bigcap_{(m_1,\ldots,m_{i-2})\in  B_1\times \cdots \times B_{i-2} } L_{i,(n,(m_j)_{j\neq i-1 })} \right|}{(\sigma \log(N))^{1/k}}  \right)\\
        &\geq  \binom{\E\left(\left| \bigcap_{(m_1,\ldots,m_{i-2})\in  B_1\times \cdots \times B_{i-2} } L_{i,(n,(m_j)_{j\neq i-1 })} \right| \right)}{(\sigma \log(N))^{1/k}} \\
        &\geq  \binom{M (\eta_1\cdots\eta_{i-2}\delta)^{(\sigma \log(N))^{(i-2)/k}}}{(\sigma \log(N))^{1/k}} \\
    \end{align*}
    where we used the induction hypothesis in the last inequality. This way, we obtain 
    \begin{align*}
        &\E\left( \left| \bigcap_{(m_1,\ldots,m_{i-1})\in  B_1\times \cdots \times B_{i-1} } L_{i,(n,(m_j)_{j\neq i })} \right| \right) \\
        &\geq  \sum_{m_i\in [M]} \binom{M}{(\sigma \log(N))^{1/k}}^{-1}\binom{M (\eta_1\cdots\eta_{i-2}\delta)^{(\sigma \log(N))^{(i-2)/k}}}{(\sigma \log(N))^{1/k}}\\
        &\geq \eta_{i-1} \sum_{m_i\in [M]} \left(\frac{M (\eta_1\cdots\eta_{i-2}\delta)^{(\sigma \log(N))^{(i-2)/k}}}{M} \right)^{(\sigma \log(N))^{1/k}}\\
        &\geq M (\eta_1\cdots\eta_{i-1}\delta)^{(\sigma \log(N))^{(i-1)/k}},
    \end{align*}
    concluding the induction. We remark that our algebraic manipulations are all valid since $M (\delta')^{(\sigma \log(N))^{(k-1)/k}}\geq\exp(\alpha \log(N)- \log(\delta'^{-1}) (\sigma\log(N))^{(k-1)/k})\geq N^{\alpha/2}$ whenever $N$ is sufficiently large.
    
    Now, we can compute the following:
    \begin{align*}
        &\E\left( \left| \bigcap_{\boldsymbol{m}\in B_1\times \cdots \times B_{k}} A_{\boldsymbol{m}} \right| \right)= N \E\left( \P\left(n\in  \bigcap_{\boldsymbol{m}\in B_1\times \cdots \times B_{k}} A_{\boldsymbol{m}} \Bigg| n\right) \right)\\
        &= N \E\left( \P\left( B_k \subseteq \bigcap_{( m_1,\ldots,m_{k-1})\in  B_1\times \cdots \times B_{k-1} } L_{k,(n,(m_j)_{j\neq k })}  \Bigg| n \right)\right)\\
        &=N \binom{M}{ (\sigma \log(N))^{1/k}}^{-1}  \E\left( \binom{\left|\bigcap_{( m_1,\ldots,m_{k-1})\in  B_1\times \cdots \times B_{k-1} } L_{k,(n,(m_j)_{j\neq k })} \right|}{ (\sigma \log(N))^{1/k}} \right)\\
        &\geq N \binom{M}{ (\sigma \log(N))^{1/k}}^{-1} \binom{  \E\left(\left|\bigcap_{( m_1,\ldots,m_{k-1})\in  B_1\times \cdots \times B_{k-1} } L_{k,(n,(m_j)_{j\neq k })} \right| \right)}{ (\sigma \log(N))^{1/k}} \\
        &\geq N \binom{M}{ (\sigma \log(N))^{1/k}}^{-1} \binom{ M(\eta_1\cdots\eta_{k-1}\delta)^{(\sigma \log(N))^{(k-1)/k}}  }{ (\sigma \log(N))^{1/k}} \\
        &\geq N ((\eta_1\cdots\eta_{k}\delta)^{(\sigma \log(N))^{(k-1)/k}} )^{ (\sigma \log(N))^{1/k}}
    \end{align*}
    where we used \cref{SIL-eq-1} with $i=k$ in the second to last inequality. Hence, we get  
    $$\E\left( \left| \bigcap_{\boldsymbol{m}\in B_1\times \cdots \times B_{k}} A_{\boldsymbol{m}} \right| \right) \geq  N (\eta_1\cdots\eta_{k}\delta)^{(\sigma \log(N)) }  = N^{1-\sigma \log(\tilde{\delta}^{-1})},  $$
    finishing the proof.
\end{proof}

\subsection{Lower Asymptotics for $\phi$}

Let $k,N\in \N$ and $\delta>0 $. Define
$$ \phi_k(\delta,N)= \min_{A\subseteq [N], |A|\geq \delta N} \max_{B_1,\ldots,B_k\subseteq [N], B_1+\cdots+B_k\subseteq A} \min\{|B_1|,\ldots,|B_k|\}.$$
That is, $\phi_k(\delta,N)$ denotes the largest integer such that any set $A\subseteq [N]$ with $|A|\geq \delta N$ contains $B_1+\cdots+B_k$ as a subset, where $B_1,\ldots,B_k\subseteq [N]$ satisfy $\min\{|B_1|,\ldots, |B_k|\} \geq \phi_k(\delta,N)$.
We will prove that for a fixed $\delta>0$, for all $\alpha<(1/\log(\delta^{-1}))^{1/(k-1)}$ the expression $\phi_k(\delta,N)/\log(N)^{1/(k-1)}$ is bounded below by $\alpha$ for all $N\in \N$ sufficiently large. Equivalently, there is $\overline{N}\in \N$ such that for all $N\geq \overline{N}$ for all $A\subseteq [N]$ with $|A|\geq \delta N$, there are $B_1,\ldots,B_k\subseteq [N]$ with $B_1+\cdots+B_k\subseteq A$ and such that $|B_1|,\ldots,|B_k|\geq \alpha \log(N)^{1/(k-1)}$.

\begin{theorem}\label{theorem-B-general-form}
    For every $\delta>0$, we have 
    $$\liminf_{N\to \infty} \frac{\phi_k(\delta,N)}{\log(N)^{1/(k-1)}} \geq \left(\frac{1}{\log(1/\delta) }\right)^{1/(k-1)}.$$
\end{theorem}
\begin{proof}
Let $\delta'<\tilde{\delta}<\delta$ and $\epsilon>0$. Let $N \in \N$ be large enough that $\delta N-(k-1)N^{1-\epsilon}\geq \tilde{\delta} N$. Let $M=N^{1-\epsilon}$ and take 
$\sigma>0$ such that
    $$\sigma< \frac{1}{\log(\delta'^{-1})} .$$
    Define for each $\boldsymbol{m}\in [M]^{k-1}$ the set $A_{\boldsymbol{m}}= (A-(m_1+\cdots+m_{k-1}))\cap [N]$. Notice that $|A_{\boldsymbol{m}}|\geq \delta N-(k-1)N^{1-\epsilon}$.  
    By \cref{FIL}, we can find sets $B_1,\ldots,B_{k-1}\subseteq [M]$ with $|B_1|,\ldots,|B_{k-1}|\geq (\sigma \log(N))^{1/(k-1)}$ such that 
    $$\left|\bigcap_{\boldsymbol{m}\in B_1\times\cdots \times B_{k-1}} A_{\boldsymbol{m}}\right|\geq N^{1-\sigma\log(\delta'^{-1})}. $$
    This implies that there are $B_1,\ldots,B_k\subseteq [N]$ with $|B_1|,\ldots,|B_k|\geq (\sigma \log(N))^{1/(k-1)}$ such that $B_1+\cdots+B_k\subseteq A$. Thus
    $$ \sigma^{1/(k-1)} \leq \liminf_{N\to \infty} \frac{\phi_k(\delta,N)}{\log(N)^{1/(k-1)}}. $$
    Taking $\sigma \nearrow  \left(\frac{1}{\log(\delta'^{-1})}\right) $ we have
    $$ \left(\frac{1}{\log(\delta'^{-1})}\right)^{1/(k-1)}\leq \liminf_{N\to \infty} \frac{\phi_k(\delta,N)}{\log(N)^{1/(k-1)}}. $$
    Taking $\delta' \nearrow \delta$ we conclude that 
        $$ \left(\frac{1}{\log(\delta^{-1})}\right)^{1/(k-1)}\leq \liminf_{N\to \infty} \frac{\phi_k(\delta,N)}{\log(N)^{1/(k-1)}}, $$
        finishing the proof.
\end{proof}

\begin{remark}
    \cref{Finite-Sumsets-Growth-rate} follows as from \cref{theorem-B-general-form} by taking $k = 2$.
\end{remark}

\small{
\bibliographystyle{abbrv}
\bibliography{refs}

}
\bigskip
\noindent
Felipe Hernández\\
\textsc{{\'E}cole Polytechnique F{\'e}d{\'e}rale de Lausanne} (EPFL)\par\nopagebreak
\noindent
\href{mailto:felipe.hernandezcastro@epfl.ch}
{\texttt{felipe.hernandezcastro@epfl.ch}}

\bigskip
\noindent
Luke Hetzel\\
\textsc{University of Denver} \par\nopagebreak
\noindent
\href{mailto:luke.hetzel@du.edu}
{\texttt{luke.hetzel@du.edu}}

\end{document}